\documentclass[reqno,a4paper,12pt]{amsart} 

\usepackage{amsmath,amscd,amsfonts,amssymb}
\usepackage{mathrsfs,dsfont}

\usepackage{tikz}
\usetikzlibrary{patterns}
\usepackage{float}
\usepackage{appendix}
\usepackage{caption}
\usepackage{subcaption}
\usepackage{enumerate}

\numberwithin{equation}{section}
\numberwithin{figure}{section}

\def\R{\mathbb{R}}

\def\Z{\mathbb{Z}}

\def\dH{\dim_{\mathcal{H}}}

\renewcommand\leq{\leqslant}
\renewcommand\geq{\geqslant}

\newcommand{\supp}{\operatorname{supp}}

\newcommand{\diam}{\operatorname{diam}}
\newcommand{\dist}{\operatorname{dist}}

\theoremstyle{plain}
\newtheorem{thm}{Theorem}[section]
\newtheorem{lem}[thm]{Lemma}

\newtheorem{prop}[thm]{Proposition}

\newtheorem{conj}[thm]{Conjecture}

\newtheorem*{claim*}{Claim}
\newtheorem*{thm*}{Theorem}

\theoremstyle{definition}
\newtheorem{definition}[thm]{Definition}
\newtheorem*{definition*}{Definition}
\newtheorem*{remarks*}{Remarks}
\newtheorem*{remark*}{Remark}

\newenvironment{enumerate-math}
{\begin{enumerate}
\addtolength{\itemsep}{5pt}
}
{\end{enumerate}}

\newenvironment{enumerate-text}
{\begin{enumerate}
\addtolength{\itemsep}{5pt}
}
{\end{enumerate}}

\begin{document}

\title{Intersection between pencils of tubes, discretized sum-product, and radial projections}

\author{Bochen Liu, Chun-Yen Shen}
\address{National Center for Theoretical Sciences, No. 1 Sec. 4 Roosevelt Rd., National Taiwan University, Taipei, 106, Taiwan}
\email{Bochen.Liu1989@gmail.com}
\email{cyshen@math.ntu.edu.tw}

\date{}


\begin{abstract}
In this paper we prove the following results in the plane. They are related to each other, while each of them has its own interest.

First we obtain an $\epsilon_0$-increment on intersection between pencils of $\delta$-tubes, under non-concentration conditions. In fact we show it is equivalent to the discretized sum-product problem, thus the $\epsilon_0$ follows from Bourgain's celebrated result.

Then we prove a couple of new results on radial projections. We also discussion about the dependence of $\epsilon_0$ and make a new conjecture.

A tube condition on Frostman measures, after careful refinement, is also given.
\end{abstract}
\maketitle

\section{Introduction}
\subsection{Radial projections}
Dimension of projections has become one of the most popular topics in geometric measure theory. It dates back to Marstrand's celebrated 1954 paper \cite{Mar54}, where he proved his well-known Marstrand projection theorem: let $\pi_e(x)=x\cdot e$, $e\in S^1$, $x\in \R^2$, denote the orthogonal projection, then for any Borel set $E\subset\R^2$,
\begin{itemize}
 	\item if $\dH E>1$, then $|\pi_e (E)|>0$ for almost all $e\in S^1$;
 	\item if $\dH E\leq 1$, $\dH \pi_e (E)=\dH E$ for almost all $e\in S^1$.
\end{itemize}

Marstrand's original proof is very complicated. In 1968, Kaufman \cite{Kau68} gave a much simpler proof via potential theory and Fourier analysis. Moreover, he obtained the sharp dimensional exponent on the exceptional set: if $\dH E\leq 1$, then \begin{equation}\label{orthogonal-exception-2}\dH\{e\in S^1 :\dH \pi_e (E)<\dH E\}\leq \dH E.\end{equation}
When $\dH E>1$, the sharp dimensional exponent on the exceptional set is due to Falconer \cite{Fal82}: if $\dH E>1$, then \begin{equation}\label{orthogonal-exception-1}\dH\{e\in S^1 :|\pi_e (E)|=0\}\leq 2-\dH E.\end{equation}
One can see \emph{Example 5.13} in \cite{Mat15}, and \cite{KM75} for sharpness examples.

It is still not clear that, given an arbitrary $\tau<\dH E$, for how many $e\in S^1$ one can expect $\dH \pi_e (E)\geq\tau$? Alternatively, given arbitrary sets $E\subset\R^2$, $\Omega\subset S^1$, how large $\max_{e\in\Omega}\dH\pi_e(E)$ can be guaranteed? This problem is far from being solved. A result of Oberlin \cite{Obe12} implies that
\begin{equation}\label{Bourgain-half}\dH \{e\in S^1: \pi_e(E)< \frac{\dH E}{2}\}=0,\end{equation}
which itself is in fact trivial (see Section \ref{app-tube-condition}). A stronger and much deeper version is due to Bourgain \cite{Bou10}: suppose $E\subset\R^2$, $\Omega\subset S^1$, $\dH E\in(0,2)$, $\dH \Omega>0$, then there exists $e\in \Omega$ such that
\begin{equation}\label{Bourgain-half-epsilon}\dH \pi_e(E)\geq \frac{\dH E}{2}+\epsilon_0(\dH E, \dH \Omega),\end{equation}
where $\epsilon_0>0$ is an absolute constant that only depends on $\dH E, \dH\Omega$. 

For any $0<\alpha<\dH E$, there are examples \cite{KM75} with 
$$\dH \{e\in S^1: \pi_e(E)< \frac{\dH E+\alpha}{2}\}=\alpha. $$Therefore \eqref{Bourgain-half} is sharp, and the dependences on both $\dH E, \dH\Omega$ (especially on $\dH \Omega$) of $\epsilon_0$ in \eqref{Bourgain-half-epsilon} are necessary.

For orthogonal projections in higher dimensions, we refer to Chapter 4, 5 in \cite{Mat15}, and \cite{He17}.

Back to Marstrand's 1954 paper. In addition to orthogonal projections, Marstrand also studied radial projections in the plane, although he did not name it radial projection and stated his results in a different way. One can also see \cite{FFJ15} for Marstrand's original wording.

For any $y\in\mathbb{R}^d$, $d\geq 2$, let $\pi^y: \mathbb{R}^d\backslash\{y\}\rightarrow S^{d-1}$ denote the radial projection
$$\pi^y(x)=\frac{x-y}{|x-y|}.$$
Marstrand \cite{Mar54} proved that, given a Borel set $E\subset\R^2$, $0<\mathcal{H}^s(E)<\infty$ for some $s>1$, then $\mathcal{H}^1(S^1\backslash \pi^y(E))=0$ for $\mathcal{H}^s$ almost all $y\in E$. People say $E$ is visible from $y$ if $\mathcal{H}^{d-1}(\pi^y(E))>0$ (see, for example, Mattila's survey \cite{Mat04}).

It is natural to compare radial projections with orthogonal projections, and see if any result above has an analog. Notice that orthogonal projections can be seen as radial projections with pins in the hyperplane at infinity.

Unlike orthogonal projections, the development on radial projections is quite slow. The sharp analog of \eqref{orthogonal-exception-1} was not known until recently, when Orponen \cite{Orp18} \cite{Orp19} proved for any Borel set $E\subset \R^d$, $\dH E>d-1$,
\begin{equation}\label{sharp-radial-projection}\dim_{\mathcal{H}}\left\{y\in\mathbb{R}^d: \mathcal{H}^{d-1}\left(\pi^y(E)\right)=0\right\}\leq 2(d-1)-\dH E.\end{equation}
Later his quantitative estimate in \cite{Orp19} played an important role in the breakthrough on the Falconer distance conjecture \cite{KS18}, \cite{GIOW18}, \cite{Shm18}, thus more attention was drawn to radial projections. 


When $\dH E\leq d-1$, the sharp analog of \eqref{orthogonal-exception-2} is still unknown in any dimension. The best known results are, given $E\subset\R^d$, $\dH E\leq d-1$,
$$\dH\{y\in\R^d: \dH \pi^y (E)<\dH E \}\leq \min\{\dH E+1,\  2(d-1)-\dH E \},$$
where the first bound follows from a general machinery of Peres and Schlag \cite{PS00}, and the second is due to the first author \cite{Liu19}. The following conjecture was proposed by the first author in \cite{Liu19}. It is generally sharp because $E$ could lie in a $k$-dimensional affine subspace. 
\begin{conj}[B.L., 2019]\label{Liu-conj}
Suppose $E\subset\R^d$ is a Borel set, $\dH E\in(k-1, k]$, $k=0,1,\dots, d-1$. Then
	$$\dH\{y\in\R^d: \dH \pi^y (E)<\dH E \}\leq k.$$
\end{conj}

In the plane, it is trivial if $E$ lies in a line, so we always assume $E$ is not  contained in a line, by which we mean $\dH (E\backslash l)= \dH E$ for any line $l$.

An analog of \eqref{Bourgain-half} was recently proved by Orponen \cite{Orp19}: suppose $E\subset\R^2$ is a Borel set, not  contained in a line, then
\begin{equation}\label{Orponen-half}\dH\{y\in\R^2: \dH \pi^y(E)< \frac{\dH E}{2}\}=0.\end{equation}
In particular the set of of directions
$$S(E):=\left\{\frac{x-y}{|x-y|}:x, y\in E, x\neq y\right\}$$
has Hausdorff dimension at least $\frac{\dH E}{2}$. The following conjecture was then made in \cite{Orp19}, Conjecture 1.9.

\begin{conj}
	[Orponen, 2019]\label{Orponen-conj}
	Suppose $E\subset$ is a Borel set in the plane, not  in a line. Then 
	$$\dH S(E):=\dH \left\{\frac{x-y}{|x-y|}:x, y\in E, x\neq y\right\}=\min\{\dH E, 1\}.$$
\end{conj}

In this paper we obtain an $\epsilon_0$-increment towards this conjecture.
\begin{thm}
	\label{main-thm-radial-proj}
	Given $0<s<2$, there exists $\epsilon_0=\epsilon_0(s)>0$ such that the following holds.

	Suppose $E\subset\R^2$ is a Borel set, $0<\dH E<2$, not  contained in a line. Then the set of directions
	$$S(E):=\left\{\frac{x-y}{|x-y|}:x, y\in E, x\neq y\right\}$$
	has Hausdorff dimension at least $\frac{\dH E}{2}+\epsilon_0(\dH E)$.
\end{thm}

We also obtain results on the pinned version. In particular it makes progress towards Conjecture \ref{Liu-conj}.
\begin{thm}
	\label{main-thm-pin-radial-proj}
	Given $0<s<2$, there exists $\epsilon_0=\epsilon_0(s)>0$ such that the following holds.

	Suppose $E, F\subset\R^2$ are Borel sets, $0<\dH E, \dH F<2$, and $E$ is not  contained in a line. Then at least one of the following happens:
\begin{enumerate}[(i)]
\item there exists $y\in F$ such that
	$$\dH \pi^y(E)\geq \frac{\dH E}{2}+\epsilon_0(\dH E).$$
\item there exists $x\in E$ such that
	$$\overline{\dim}_{\mathcal{M}}\, \pi^x(F)= \dH F.$$
\end{enumerate}

In particular,
$$\dH\left\{y\in\R^2: \dH \pi^y(E)<\frac{\dH E}{2}+\epsilon_0(\dH E)\right\}\leq 1.$$
\end{thm}

We would like to emphasis that the $\epsilon_0$ here only depends on $\dH E$, not on $\dH F$. This does not contradict examples from orthogonal projections on \eqref{Bourgain-half-epsilon} under projective transformations, in which case $F$ must lie in a line thus $(ii)$ always holds. This also shows the dichotomy in Theorem \ref{main-thm-pin-radial-proj} is necessary.

Theorem \ref{main-thm-pin-radial-proj} suggests that, the behavior of radial projections is expected to be much better if $F$ is not in a line. We remind the reader that the sharpness of \eqref{sharp-radial-projection} follows from examples with exceptional sets in a hyperplane \cite{Orp18}. Also we do not know any other counterexample on Conjecture \ref{Liu-conj} except affine subspaces. With these in mind we would like to make a wild guess. It is stronger than both Conjecture \ref{Liu-conj} and Conjecture \ref{Orponen-conj} above.
\begin{conj}
	Suppose $E, F$ are Borel sets in the plane, $\dH E, \dH F>0$, $F$ not in a line. Then there exists $y\in F$ such that
	$$\dH \pi^y(E)=\min\{\dH E, 1\}.$$
\end{conj}
One may also expect $\min\{\dH E, d-1\}$ in $\R^d$, $d\geq 2$, given $F$ not in any proper affine subspace. Any progress or counterexample would be appreciated. For example we would be very happy to see if Theorem \ref{main-thm-pin-radial-proj} can be improved to
$$\dH \pi^y(E)\geq \frac{\dH E}{2}+\epsilon_0(\dH E),\ y\in F,$$
given $0<\dH E<2$ and $\dH F>0$ not in a line.

\subsection{Pencils of tubes} To prove our theorems on radial projections, the key new ingredient is a non-trivial estimate on intersection between pencils of $\delta$-tubes. These days intersection between $\delta$-tubes has attracted a lot of interests, especially from harmonic analysts, so we hope that this structure would shed lights on other problems in the future.

We start from intersection between pencils of lines in incidence geometry. An $n$-pencil in the plane, with tip $p\in\R^2$, is a set of $n$ concurrent lines passing through $p$. Given $m$ $n$-pencils, an $m$-rich point is a point passed through by one line from each pencil. The following question was raised by Rudnev:
\begin{center}
	\textit{Given $m$ $n$-pencils in the plane, how large the set of m-rich points can be?}
\end{center}

By projective transformations, one can also take parallel lines into account, that can be seen as pencils of lines with tips in the infinity line.

We may assume $m\ll n$ and no line is shared by every pencil. Then the trivial upper bound is $C_m n^2$, which can be attained when tips lie in the same line, and also $m=3$ with non-colinear tips. See Figure \ref{eg-pencil-line} below for projective images of these examples.
\begin{figure}[H]
\centering
\begin{subfigure}[b]{.5\textwidth}
\centering
\begin{tikzpicture}[scale=1, p2/.style={line width=0.275, black}, p3/.style={line width=0.15, black!50!white}]
\foreach \x in {1,...,4}{
\draw (\x, -0.25) node[anchor=north]{\x};
\draw (-0.25, \x) node[anchor=east]{\x};}
\foreach \x in {0,...,4}{
\draw (\x,-0.25) -- (\x, 4.5);
\foreach \y in {0,...,4}{
\draw (-0.25, \y) -- (4.5, \y);
	\fill (\x,\y) circle (0.1); }}
\foreach \m in {1,...,5}
\draw (\m+0.2, -0.2) -- (-0.2, \m+0.2);
\foreach \m in {0,...,2}
\draw (\m-0.4, -0.2) -- (5, 2.5-\m/2);
\foreach \m in {-4,...,0}
\draw (-0.4, -0.2-\m/2) -- (5, 2.5-\m/2);
\end{tikzpicture}
\caption*{$m\ll n$ with colinear tips (parallel lines)}
\end{subfigure}%
\begin{subfigure}[b]{.5\textwidth}
\centering
\begin{tikzpicture}[scale=0.26, p2/.style={line width=0.275, black}, p3/.style={line width=0.15, black!50!white}]
\draw (0,0) node[anchor=north east]{0};
\foreach \x in {1,...,4}{
\draw (2^\x,0) -- (2^\x, 2^4+2);
\draw (2^\x,0) node[anchor=north]{$2^\x$};
\foreach \y in {1,...,4}{
\draw (0, 2^\y) -- (2^4+2, 2^\y);
\draw (0, 2^\y) node[anchor=east]{$2^\y$};
	\fill (2^\x,2^\y) circle (0.4); 
	\draw (0,0) -- (1.1*2^\x, 1.1*2^\y); }}
\end{tikzpicture}
\caption*{$m=3$ with non-colinear tips}
\end{subfigure}
\caption{Projective images of counterexamples on pencils of lines}\label{eg-pencil-line}
\end{figure}
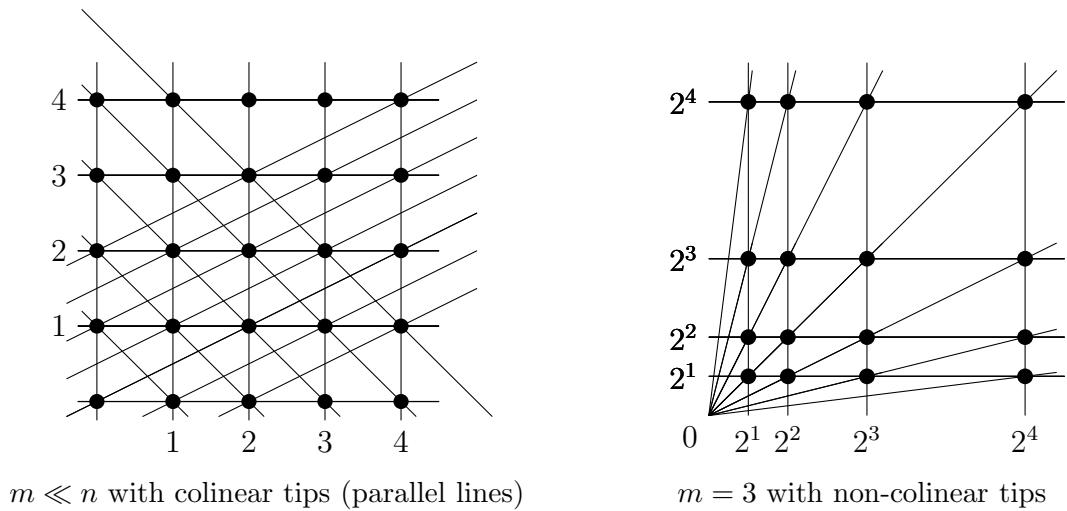

Now the simplest remaining case is $m=4$ with tips non-colinear, and indeed non-trivial exponents follow. A result of Chang and Solymosi \cite{CS07} implies that the number of $4$-rich points is $O(n^{2-1/24})$, which was improved to $O(n^{2-1/6})$ by Roche-Newton and Warren \cite{RW18}. No better exponent is known yet for $m>4$. On the other hand, one cannot beat $C n^{3/2}/m^3$ for any $m\geq 4$ \cite{ARS18} \cite{RW18}, even under the assumption that no three tips lie in the same line.

In this paper we consider the $\delta$-neighborhood of pencils of lines, namely pencils of $\delta$-tubes. We denote these pencils of $\delta$-tubes by $P_i$ and ask how large 
$$\left|\bigcap_{1\leq i\leq m} P_i\cap[0,1]^2\right|$$
could be. Examples as above show that the trivial upper bound is $C \delta^2n^2$ and it is interesting only if $m\geq 4$ with tips non-colinear. However, these are still not enough for non-trivial exponents, because $P_i$ could be a sector of angle $n\delta$. Therefore we need a non-concentration condition on directions of these tubes. This dates back to Katz and Tao \cite{KT01}.

\begin{definition}
	\label{def-KT}
	We say $A\subset\R^d$ is a $(\delta, \sigma)_d$-set if
$$|A \cap B(a,r)|\lessapprox \delta^d\left(r/\delta\right)^\sigma$$
for any $a\in\R^d$ and $\delta<r<1$. For convenience we write $(\delta, \sigma)$ when $d=1$.
\end{definition}

We may identify $S^1$ with $[0,1)$. We say a pencil of tubes $P$ with tip $p\in\R^2$ is a $(\delta, \sigma)$-pencil, if the direction set $\pi^{p}(P)$ is a $(\delta, \sigma)$-set. Our result is the following.
\begin{thm}
	\label{main-pencil}
	For any $0<\sigma<1$, there exist $\epsilon_0=\epsilon_0(\sigma)>0$ such that the following holds.

	Suppose $P_i$, $i=1,2,3,4$, are $(\delta, \sigma)$-pencils with tips $p_i$, $i=1,2,3,4$, satisfying
	\begin{itemize}		
		\item $\dist(p_i, p_j)\approx\dist(p_i, [0,1]^2)\approx\dist(l_{p_i, p_j}, [0,1]^2)\approx 1$, $i, j=1,2,3,4$, $i\neq j$;
		\item tips $p_i$, $i=1,2,3,4$, are not contained in a tube of radius $\approx1$.
	\end{itemize}		
	Then
	$$\left|\bigcap_{1\leq i\leq 4} P_i\cap [0,1]^2 \right| \lesssim \delta^{2-2\sigma+\epsilon_0}. $$ 
\end{thm}

Similar to the discrete case, by projective transformations one can also take parallel tubes into account.

\subsection{Discretized sum-product}
Sum-product is a fundamental phenomenon in mathematics. Roughly speaking it says a subset cannot be additive and multiplicative simultaneously, where the ambient space could be $\R$, $\Z$, finite fields, etc. There is a very large body of literature on this problem, but it is still far from being completely understood.

In this paper we consider the $\delta$-discretized version. It was raised by Katz and Tao \cite{KT01} and their conjecture was solved by Bourgain \cite{Bou03}.
\begin{thm}[Bourgain, 2003]\label{Bourgain-epsilon}
	Given $0<\sigma<1$, there exists $\epsilon_0>0$ such that, for any $(\delta, \sigma)$-set $A\subset [0,1]$, $|A|\approx\delta^{1-\sigma}$, we have
		$$\max \{|A+A|, |AA| \} \gtrsim \delta^{-\epsilon_0} |A|.$$
\end{thm}
An explicit expression of $\epsilon_0$ was recently given by Guth, Katz and Zahl \cite{GKZ18} (see also \cite{Che19}). The existence of $\epsilon_0$ still holds under milder conditions (see, e.g. \cite{BG08}, \cite{Bou10}), and all results in this paper still hold accordingly.

Discretized sum-product has proved connected to Borel rings in the real, distance sets, Furstenburg sets, orthogonal projections, spectral gaps, Besicovitch sets (see, for example, \cite{KT01} \cite{Bou03} \cite{BG08} \cite{Bou10} \cite{BG12} \cite{BD17} \cite{KZ19}), and many others. In this paper we add pencils of tubes and radial projections into this family.

\begin{thm}
	\label{main-theorem-equivalence}
	The following statements are equivalent.
	\begin{enumerate}[(i)]
	    \item Theorem \ref{Bourgain-epsilon};
		\item Theorem \ref{main-pencil};
		\item Given $0<\sigma<1$, there exists $\epsilon_0>0$ such that, for any $(\delta, \sigma)$-set $A\subset [1/4,1/2]$, $|A|\approx\delta^{1-\sigma}$, we have
		$$\max\{|AA|, |(1-A)(1-A)|\}\gtrsim \delta^{-\epsilon_0}|A|.$$
	\end{enumerate}
\end{thm}

\subsection*{Organization.} This paper is organized as follows. In Section 2 we review some useful lemmas from additive combinatorics and geometric measure theory. In Section 3, 4, 5 we prove Theorem \ref{main-theorem-equivalence}, thus Theorem \ref{main-pencil} follows. In Section 6 we prove a tube condition on Frostman measures after careful refinement. In Section 7, 8 we prove Theorem \ref{main-thm-radial-proj}, \ref{main-thm-pin-radial-proj}. In Section 9 we reprove \eqref{Orponen-half} in several lines with our tube condition.

\subsection*{Notation}
\ 

$|E|$ denotes the Lebesgue measure of a Borel set $E$; $\#(A)$ denotes the cardinality of a finite set $A$.

$\dH$ denotes the Hausdorff dimension; $\mathcal{H}^s$ denotes the $s$-dimensional Hausdorff measure.

$X \lesssim Y$ means that $X\leq CY$ for a constant $C>0$. $X\lessapprox Y$ means for any $\epsilon>0$ there exists $C_\epsilon>0$ such that $X\leq C_\epsilon \delta^{-\epsilon} Y$. $X\approx Y$ means $X\lessapprox Y\lessapprox X$.

$\delta_k=\delta_k(\epsilon)$ denotes the hyperdyadic number $2^{-(1+\epsilon)^k}$.

Given $x, y\in \R^2$, $x\neq y$ let $l_{x,y}$ denote the line passing through $x,y$ and $l_y$ denote an arbitrary line passing through $y$. Let $P_y$ denote a pencil with tip $y$. 

$T(l, \delta)$ denotes the $\delta$-neighborhood of $l$. We say two lines are $\delta^\rho$-separated if they make an angle $>\delta^\rho$. We say two tubes are are $\delta^\rho$-separated if their central lines are $\delta^\rho$-separated.

\subsection*{Acknowledgments} We would like to thank Wei-Hsuan Yu for introducing us to each other.

\section{Preliminaries}
\subsection{Additive combinatorics}
\ 

The first lemma is one of fundamental tools in additive combinatorics. One can see, for example, Theorem 2.29 in \cite{TV06}.
\begin{lem}
	[Balog-Szemer\'edi-Gowers theorem] Suppose $A, B$ are finite subsets of an additive group $Z$, $G\subset A\times B$, $K>1$, satisfying $$\#(G)>\#(A)\#(B)/K$$ and
	$$\#(A\overset{G}{+}B):=\#(\{x+y: (x,y)\in G\})\leq K\#(A)^{1/2}\#(B)^{1/2}.$$
	Then there exist $A'\subset A$, $B'\subset B$ such that
	$$\#(A')\gtrsim\#(A)/K;$$
	$$\#(B')\gtrsim \#(B)/K;$$
	$$\#(A'+B')\lesssim K^{O(1)}\#(A)^{1/2}\#(B)^{1/2}.$$
\end{lem}

In the proof of $(ii)\implies(iii)$ in Theorem \ref{main-theorem-equivalence}, we shall use the following refinement lemma due to Katz and Tao.
\begin{lem}[Refinement 2.2 in \cite{KT01}]\label{Katz-Tao-Refinement}
Let $0<\delta \ll 1$ be a dyadic number, $ 0 < \sigma < 1$, $K\gg 1$ be a constant, and $A\subset [0,1]$ be a union of $\delta$-intervals with $|A| \lessapprox \delta^{1 - \sigma}$. Then one can find a set $A_{\delta '}$ for all dyadic numbers $\delta < \delta ' \leq 1$ which can be covered 
by $\lessapprox \delta^{K\epsilon}\delta'^{-\sigma}$ balls of radius $\delta '$, and a set $(\delta, \sigma)$-set $A^{*}$ contained in the $\delta$-neighborhood of $A$, such that
$$A \subset A^{*} \cup \bigcup_{\delta < \delta ' \leq 1} A_{\delta'}.$$
\end{lem}
\begin{proof}
	Let
	$$A_{\delta'}:=\{x\in\R: |A\cap B(x, \delta')|\geq \delta^{-K\epsilon}\delta(\delta'/\delta)^{\sigma}\}$$
	and $A^*$ be the $\delta$-neighborhood of
	$$A\backslash \bigcup_{\delta<\delta'\leq 1} A_{\delta'}.$$
\end{proof}

\subsection{Geometric measure theory}\ 

There are natural probability measures on Borel sets in $\R^d$ that demonstrate its Hausdorff dimension. See, for example, Section 2.5 in \cite{Mat15}.
\begin{lem}
	[Frostman Lemma]
	Suppose $E$ is a Borel set in $\R^d$. Then for any $0\leq s<\dH E$ there exists a probability measure $\mu$ on $E$ such that for any $x\in\R^d$, $r>0$,
	$$\mu(B(x,r))\lesssim r^s. $$
\end{lem}

The technical reason why we need $E$ not in a line is the following.
\begin{lem}[Lemma 2.1 in \cite{Orp19}]\label{lemma-not-in-line}
Assume that $\mu$ is a Borel probability measure on the unit ball, and $\mu(l)=0$ for all lines $l$. Then, for any $\epsilon>0$, there exists $\delta>0$ such that $\mu(T(l,\delta))\leq\epsilon$ for any line $l$.	
\end{lem}
\begin{proof}
	[Sketch of the proof] If it fails for some $\epsilon>0$, there exists a sequence of lines $l_j$ and a line $l$ such that $\mu(l_j)\geq\epsilon$, $l_j\rightarrow l$ under Hausdorff metric. This implies that there exists a sequence $\delta_j\searrow 0$ such that $\mu(T(l, \delta_j))\geq\epsilon $, and therefore $\mu(l)\geq\epsilon$, contradiction.
\end{proof}

We learn the following discretization of fractals from Lemma 7.5 in \cite{KT01}. We state a different version.
\begin{lem}\label{discretize-fractal}
Let $A\subset\R^d$ be compact. If $\dH (A)<\sigma$, then there exist $\epsilon>0$ such that, for any $k_0>0$, there exists a family of $(\delta_k(\epsilon), \sigma)_d$-sets $X^\sigma_{k}$, $k\geq k_0$, that cover $A$.
\end{lem}
\begin{proof}
	[Sketch of the proof] There exist $\epsilon, \epsilon'>0$ such that, for every $k_0$ large enough, one can find a collection of hyperdyadic balls $B_{r_i}$, $r_i\leq \delta_{k_0}(\epsilon)$, that cover $A$, and 
	$$\sum r_i^{\sigma-\epsilon'}\ll 1.$$
	For each hyperdyadic number $r$, denote
	$$Y_{r}:=\bigcup_{r_i=r} B_{r_i}$$
	and let $\bf{Q}_r$ be a collection of hyperdyadic cubes $Q$ of side-length at least $r$ that cover $Y_r$, and minimize the quantity
	$$\sum_{Q\in \bf{Q}_r} \diam (Q)^\sigma.$$

	Finally we take $$X^\sigma_k=\bigcup_{r}\bigcup_{Q\in \bf{Q}_r,\, \diam(Q)=\delta_k} Q +B(0, \delta_k).$$
Notice the minimality implies the non-concentration condition, and $\epsilon'>0$ guarantees that each $k$ is associated to finitely many $r$.

	For a detailed proof, we refer to Section 7 in \cite{KT01}. We point out a typo there that the assumption is supposed to be $\dH (A)<\alpha$, and the exponent in (38) in \cite{KT01} should be $\alpha-C\epsilon$. Then in the end $\delta^{\alpha/\epsilon}\lesssim r<c$, which guarantees $\lesssim\log(1/\delta)^2$ options for hyperdyadic numbers $r, c<1$ associated to each $\delta$.
\end{proof}

\begin{lem}[A pigeonholing lemma]\label{pigeonholing-lemma}
	Suppose $(X, \mu)$ is a measurable space, $\mu$ is a finite measure, and $X_1,\dots,X_M$ are measurable subsets of $X$ such that $\mu(X_i)\geq \lambda\mu(X)$, $M\lambda>2$. Then there exist $1\leq i<j\leq M$ such that
	$$\mu(X_i\cap X_j)\geq \frac{\lambda^2}{2}\,\mu(X). $$
\end{lem}
\begin{proof}
	Notice
	$$M\lambda\mu(X)\sum\mu(X_i)\leq\left(\sum\mu(X_i)\right)^2=\left(\int_X\sum 1_{X_i} d\mu \right)^2\leq \mu(X)\cdot\int\left|\sum 1_{X_i}\right|^2\,d\mu.$$
	Write
	$$\int\left|\sum 1_{X_i}\right|^2\,d\mu=\sum\mu(X_i)+\sum_{i\neq j}\mu(X_i\cap X_j). $$
	Hence
	$$\sum_{i\neq j}\mu(X_i\cap X_j)\geq (M\lambda-1)\sum\mu(X_i)\geq M\lambda(M\lambda-1)\mu(X)$$
	and the lemma follows by pigeonholing.

\end{proof}

\section{Proof of Theorem \ref{main-theorem-equivalence}: $(i)\implies(ii)$}
Suppose $(ii)$ fails, namely
\begin{equation}\label{ii-fail}\left|\bigcap_{1\leq i\leq 4} P_i\cap [0,1]^2 \right| \approx \delta^{2-2\sigma}.\end{equation}
It implies that $|\pi^{p_i}(P_i)|\approx \delta^{1-\sigma}$, $i=1,2,3,4$.

By projective transformations, we may assume $P_1=A\times\R$, $P_2=\R\times B$, $p_3=(0,0)$, $p_4=(t_0, 1)$, where $t_0\in[-1,1]$, $A, B\subset[1/4,1/2]$ are $(\delta,\sigma)$-sets, $|A|\approx|B|\approx\delta^{1-\sigma}$. 

Let us focus on $P_3$ first. Discretize $A, B$ to
$$A_d=\{a\in\delta\Z: \dist(e^a, A)<\delta\},$$ $$B_d=\{b\in\delta\Z: \dist(e^b, B)<\delta\},$$
then $\#(A_d)\approx\#(B_d)\approx\delta^{-\sigma}$.

Let 
$$G=\{(a,b)\in A_d\times B_d: \dist((e^a, e^b), P_3)<\delta\}$$
By \eqref{ii-fail}, $\#(G)\approx \delta^{-2\sigma}$ and $\#(A_d\overset{G}{-} B_d)\approx \delta^{-\sigma}$.
 Then we apply the Balog-Szemer\'edi-Gowers theorem to obtain $A'_d\subset A_d$, $B'_d\subset B_d$ with $$\#(A'_d)\approx \#(B'_d)\approx\#(A'_d-B'_d)\approx \delta^{-\sigma},$$ and then corresponding $A'\subset A$, $B'\subset B$ with
 $$|A'|\approx|B'|\approx|A'/B'|\approx\delta^{1-\sigma}.$$

It is similar to work with $P_4$. Discretize $A', B'$ to
$$A'_{d'}=\{a\in\delta\Z: \dist(t_0-e^a, A')<\delta\},$$ $$B'_{d'}=\{b\in\delta\Z: \dist(1-e^b, B')<\delta\},$$
then $\#(A'_{d'})\approx\#(B'_{d'})\approx\delta^{-\sigma}$.

Let 
$$G'=\{(a,b)\in A'_{d'}\times B'_{d'}: \dist((t_0-e^a, 1-e^b), P_4)<\delta\}$$
By \eqref{ii-fail}, $\#(G')\approx \delta^{-2\sigma}$ and $\#(A'_{d'}\overset{G'}{-} B'_{d'})\approx \delta^{-\sigma}$.
 Then we apply the Balog-Szemer\'edi-Gowers theorem to obtain $A''_{d'}\subset A'_{d'}$, $B''_{d'}\subset B'_{d'}$, $$\#(A''_{d'})\approx \#(B'_{d'})\approx\#(A''_{d'}-B''_{d'})\approx \delta^{-\sigma}.$$
 By Ruzsa triangle inequality we also have $\#(B''_{d'}-B''_{d'})\approx \delta^{-\sigma}$. Therefore there exist $A''\subset A'$, $B''\subset B'$ with
 $$|A''|\approx|B''|\approx|(t_0-A'')/(1-B'')|\approx|(1-B'')/(1-B'')|\approx\delta^{1-\sigma}.$$
 Also 
 $$\delta^{1-\sigma}\approx|A''|\leq|A''/B''|\leq |A'/B'|\approx\delta^{1-\sigma}.$$

For convenience, from now we denote $A=A''$, $B=B''$. Then
$$|A|\approx |B|\approx |A/B|\approx |(t_0-A)/(1-B)|\approx \left|(1-B)/(1-B)\right|\approx\delta^{1-\sigma}.$$

We shall show it contradicts Theorem \ref{Bourgain-epsilon}. Denote
$$X:=\left\{(b, b', x)\in B\times B\times A/B: \frac{t_0-bx}{1-b'}\in \frac{t_0-A}{1-B}\right\}.$$

Notice for any fixed $b, b'\in B$ and any $a\in A$, one can take $x=a/b$ to have
$$\frac{t_0-bx}{1-b'}=\frac{t_0-a}{1-b'}\in\frac{t_0-A}{1-B}.$$

Therefore
$$|X|\approx \delta^{3-3\sigma}.$$

Write
$$\frac{t_0-bx}{1-b'}= \frac{t_0}{1-b'}-\left(\frac{1}{1-b'}-\frac{1-b}{1-b'}\right)x$$
and change variables
$$(u,v)=\left(\frac{1-b}{1-b'}, \frac{1}{1-b'}\right)\in \frac{1-B}{1-B}\times\frac{1}{1-B}.$$

It follows that
$$\left|\left\{(u, v, x)\in \frac{1-B}{1-B}\times \frac{1}{1-B}\times A/B: t_0v-(v-u)x\in \frac{t_0-A}{1-B}\right\}\right|\approx|X|\approx\delta^{3-3\sigma}. $$

Since $|A/B|\approx \delta^{1-\sigma}$, there exists $x_0\in A/B$, $x_0\neq t_0$, such that
$$\left|\left\{(u, v)\in \frac{1-B}{1-B}\times \frac{1}{1-B}: t_0v-(v-u)x_0\in \frac{t_0-A}{1-B}\right\}\right|\gtrapprox\delta^{2-2\sigma}\approx\left|\frac{1-B}{1-B}\right|\cdot\left| \frac{1}{1-B}\right|. $$

Denote the set in the left hand side by $G$ and write $$t_0v-(v-u)x_0=(t_0-x_0)v+x_0u.$$

Since $G$ is a refinement of $(1-B)/(1-B)\times 1/(1-B)$, by the Balog-Szemer\'edi-Gowers theorem there exists $C\subset (1-B)/(1-B)$, $D\subset 1/(1-B)$ such that
$$|C|\approx|D|\approx|(t_0-x_0)D+x_0C|\approx \delta^{1-\sigma},$$
which implies $|D+D|\approx \delta^{1-\sigma}$. On the other hand
$$|DD|\leq |(1-B)(1-B)|\lessapprox\delta^{1-\sigma},$$
a contradiction to Theorem \ref{Bourgain-epsilon}.


\section{Proof of Theorem \ref{main-theorem-equivalence}: $(ii)\implies(iii)$}
Suppose $(iii)$ fails:
$$\max\{|AA|, |(1-A)(1-A)|\}\approx |A|\approx\delta^{1-\sigma}.$$
Consider $4$ pencils of tubes: 
\begin{itemize}
	\item $P_1$: vertical tubes $A\times\R$;
	\item $P_2$: horizontal tubes $\R\times A$;
	\item $P_3$: $\R(A\times A)$, with tip $p_3=(0,0)$;
	\item $P_4$: $\R(1-A)(1-A)+(1, 1)$, with tip $p_4=(1,1)$.
\end{itemize}
Notice $A\times A$ is contained in the intersection of these pencils.

To apply Theorem \ref{main-pencil}, it remains to check the non-concentration condition. It holds on $P_1$ and $P_2$, while not guaranteed on others. Therefore we apply Lemma \ref{Katz-Tao-Refinement} to direction sets $\pi^{p_i}(P_i)$, $i=3,4$ to obtain $(\delta, \sigma)$-pencils $P_3^*$, $P_4^*$, contained in the $\delta$-neighborhood of $P_3$, $P_4$, such that for any dyadic number $\delta<\delta'\leq 1$, there exist $(P_i)_{\delta'}$, $i=3,4$, that can be covered by $\delta^{K\epsilon}\delta'^{-\sigma}$ tubes of radius $\delta'$, and
$$P_i\subset P_i^*\cup \bigcup_{\delta<\delta'\leq\delta^{\epsilon_1}}(P_i)_{\delta'},\ i=3,4. $$

By the non-concentration condition on $A$, for each $\delta'$-tube $T$,
$$|T\cap (A\times A)|\lesssim_{\epsilon} \delta^{-\epsilon}{\delta'}^\sigma |A|^2.$$
Therefore by Lemma \ref{Katz-Tao-Refinement}
$$\left|\bigcup_{\delta<\delta'\leq 1} (P_i)_{\delta'}\cap (A\times A)\right|\lesssim_\epsilon \sum_{\delta<\delta'\leq 1} \delta^{(K-1)\epsilon}\delta'^{-\sigma}\cdot \delta'^\sigma |A|^2\lesssim_\epsilon\delta^{(K-2)\epsilon}|A|^2,$$
negligible to $A\times A$. Hence by Theorem \ref{main-pencil},
$$\delta^{2-2\sigma}\approx|A|^2\lesssim \left|\bigcap_{i=1}^4 P_i\right|\approx \left|P_1\cap P_2\cap P_3^*\cap P_4^*\right|\lesssim \delta^{2-2\sigma+\epsilon_0},$$
contradiction.

\section{Proof of Theorem \ref{main-theorem-equivalence}: $(iii)\implies(i)$}
Say $|A+A|=K|A|$. Since $\chi_A*\chi_A$ is supported on $A+A$ and
$$|A|^2=\int \chi_A*\chi_A,$$
there exists $z\in A+A\subset [1/2, 1]$ such that
$$|(z-A)\cap A|=\chi_A*\chi_A(z) \geq \frac{|A|^2}{|A+A|}=|A|/K. $$

Take $A_z= \frac{A}{z} \cap (1- \frac{A}{z})$. By $(iii)$ in Theorem \ref{main-theorem-equivalence}, 
$$\max \{|A_zA_z|, |(1-A_z)(1-A_z)|\} \gtrsim \delta^{-\epsilon_0} |A_z|\gtrsim \delta^{-\epsilon_0} |A|/K.$$

Since by our construction both 
$$A_zA_z,\  (1-A_z)(1-A_z) \subset \frac{A \cdot A}{z^2},$$
it follows that
$$|AA|\gtrsim \delta^{-\epsilon_0} |A|/K.$$
Finally we choose $K$ to obtain
$$\max \{|A+A|, |AA|\} \gtrsim \delta^{-\epsilon_0/2}\,|A|.$$

\section{A tube condition on Frostman measures}
In many cases people need to control measures of $\delta$-tubes. Unfortunately on Frostman measures only ball condition is given, while no tube condition is generally guaranteed. The following proposition shows, given two Frostman measures $\mu$, $\nu$ not in a line, one can always refine $\supp \mu\times\supp \nu$ carefully to ensure each tube determined by remaining pairs $(x, y)\in\supp\mu\times\supp \nu$ is under control. 

\begin{prop}\label{prop-tube-condition}
Suppose $\mu$, $\nu$ are probability measures on disjoint compact set $E, F\subset \R^2$ respectively, $\mu(l)=\nu(l)=0$ for any line $l$, and there exist $s_\mu, s_\nu\in[0,2]$, $c_\mu, c_\nu>0$ such that for any $x\in\R^2$, $r>0$,
$$\mu(B(x, r))\leq c_{\mu} r^{s_\mu},\ \nu(B(x, r))\leq c_{\nu} r^{s_\nu}.$$
Then there exists $\kappa=\kappa(s_\mu, s_\nu)>0$, a compact set $G\subset E\times F$, $\mu\times\nu (G)>0$, and a constant $C=C(\mu,\nu)>0$, such that for any line $l_{x,y}$ determined by $(x, y)\in G$,
$$\mu\left(T(l_{x,y}, r)\right), \,\nu\left(T(l_{x,y}, r)\right)\leq C r^{\kappa},\ \forall\, r>0.$$
\end{prop}

The proof is inspired by Orponen's argument on \eqref{Orponen-half}. We also simplify his framework. Although it somewhat looks weaker than that in \cite{Orp19} (see Lemma 2.2, 2.3 there), it is not less powerful. In Section \ref{app-tube-condition} we shall see that \eqref{Orponen-half} easily follows from Proposition \ref{prop-tube-condition}. Generally speaking, if $x_1, x_2\in E$, $y\in F$, or $x\in E$, $y_1, y_2\in F$, lie in a tube, we apply Proposition \ref{prop-tube-condition}, otherwise there is transversality. This tube condition may have its own interest. 

As a remark, improvement on the value of $\kappa$ does not improve anything else in this paper. In contrast, when $\dH E, \dH F>1$, the quantitative estimate on \eqref{sharp-radial-projection} in \cite{Orp19} gives a tube condition for any $\kappa<1$, whose value does matter in recent work on the distance problem (see \cite{KS18}, \cite{GIOW18}, \cite{Shm18}).  

\begin{proof}
	[Proof of Proposition \ref{prop-tube-condition}]
Let $\eta=\eta(s_\mu, s_\nu)>0$, $\rho=\rho(s_\mu, s_\nu)>0$ be positive constants that will be determined later. It suffices to find $k_0=k_0(\mu, \nu)>0$ such that $$\mu\left(T(l_{x,y}, \delta_k)\right)\leq \delta_k^{\eta},\ k=k_0,k_0+1,\dots.$$
Similarly it holds on $\nu$.

	Throughout this proof, denote $\delta=\delta_{k}$.  We say a $\delta$-tube $T$ is bad if $\mu(T)>\delta^{\eta}$ and a point $y\in F$ is bad if there exists a bad tube $T(l_y, \delta)$.
Denote by $\textbf{Bad}_\textbf{P}$ the set of bad points in $F$ and by $\textbf{Bad}_\textbf{T}$ the set of bad $\delta$-tubes passing through points in $\textbf{Bad}_\textbf{P}$.

\subsection{}\label{tube-condition-step-1}We shall find a subset $\textbf{M-Bad}_\textbf{T}\subset \textbf{Bad}_\textbf{T}$ and a small number ${\rho}>0$ such that
\begin{enumerate}[(1)]
\item $\#(\textbf{M-Bad}_\textbf{T})\leq 2\delta^{-{\eta}}$;
\item for any $y\in \textbf{Bad}_\textbf{P}$ and any bad tube $T(l_y, \delta)$, there exists (at least one) $T(l_{y, M},\delta)\in \textbf{M-Bad}_\textbf{T}$ such that \begin{enumerate}[(a)]\item $y\in T(l_{y, M}, \delta^{\rho})$, and \item  the angle between central lines $l_{y, M}$ and $l_y$ is $<\delta^{\rho}$.
\end{enumerate}
\end{enumerate}

For convenience we denote $T_y=T(l_y,\delta)$ and $T_{y, M}=T(l_{y, M},\delta)$. We say $y$ and $T=T_{y, M}$ are associated, , denoted by $y\sim T$, if there exists a bad tube $T_y$ such that the (2a), (2b) hold.

Now we construct $\textbf{M-Bad}_\textbf{T}$. Take $\textbf{M-Bad}_\textbf{T}\subset \textbf{Bad}_\textbf{T}$ as a maximal subset such that for any $T, T'\in \textbf{M-Bad}_\textbf{T}$, $T\neq T'$,
$$\mu(T\cap T')< \delta^{2{\eta}}/2. $$
By the pigeonholing lemma (Lemma \ref{pigeonholing-lemma}),
$$\#(\textbf{M-Bad}_\textbf{T})\leq 2\delta^{-{\eta}}. $$

Now we check (2a) and (2b). By the maximality, it follows that, for any $y\in \textbf{Bad}_\textbf{P}$ and any bad tube $T(l_y, \delta)\in \textbf{Bad}_\textbf{T}$, there exists $T_{y, M}\in \textbf{M-Bad}_\textbf{T}$ such that
\begin{equation}\label{Maximality-MBad-P}\mu(T_y\cap T_{y, M})\geq\delta^{2{\eta}}/2.\end{equation}

If (2b) fails, the intersection between $T_{y, M}$ and $T_y$ is contained in a ball of radius $\delta^{1-{\rho}}$. By the ball condition on $\mu$,
	$$\mu (T_y\cap T_{y, M})\leq c_\mu \delta^{s_\mu(1-{\rho})},$$
	which is $< \delta^{2{\eta}}/2$, thus contradicts \eqref{Maximality-MBad-P}, if $\eta, \rho$ are chosen to satisfy
	\begin{equation}\label{const-track-1-1}s_\mu(1-{\rho})>2{\eta}.\end{equation}

Similarly, if (2a) fails, then $l_{y, M}$ and any $l_y$ are $\delta^{\rho}$-separated, which leads to the same contradiction. 

\subsection{}\label{tube-condition-step-2}
Remove a negligible subset $E\times\textbf{BadBad}_\textbf{P}$ from $E\times F$.

From \ref{tube-condition-step-1} each $y\in \textbf{Bad}_\textbf{P}$ is associated with one (or more) $T_{y, M}\in \textbf{M-Bad}_\textbf{T}$ and $\#(\textbf{M-Bad}_\textbf{T})\leq 2\delta^{-{\eta}}$. Denote by $\textbf{BadBad}_\textbf{P}$ the set of $y\in \textbf{Bad}_\textbf{P}$ that is associated with two tubes in $\textbf{M-Bad}_\textbf{T}$ that are $\delta^{{\rho}/2}$-separated, say $T_{y, M_1}$ and $T_{y, M_2}$. 

To see it is negligible, by (2a) from \ref{tube-condition-step-1}, for any $y\in \textbf{BadBad}_\textbf{P}$,
$$y\in T(l_{y, M_1}, \delta^{\rho}) \cap T(l_{y, M_2}, \delta^{\rho})$$
contained in a ball of radius $\delta^{{\rho}/2}$. Then by $(1)$ from Subsection \ref{tube-condition-step-1} and the ball condition on $\nu$,
$$\nu(\textbf{BadBad}_\textbf{P})\leq 4c_\nu \,\delta^{-2{\eta}+s_\nu{\rho}/2}.$$
Therefore $E\times\textbf{BadBad}_\textbf{P}$ is negligible to $E\times F$, even after taking sum over $k\geq k_0$, if $k_0$ is large enough and $\eta, \rho$ are chosen to satisfy
\begin{equation}\label{const-track-1-2}
-2{\eta}+s_\nu{\rho}/2>0.
\end{equation}

\subsection{}\label{tube-condition-step 3}
Finally we remove 
$$H_k:=\left\{T(l_{y, M},\delta^{\rho/2})\times\{y\}: y\in F\backslash\textbf{BadBad}_\textbf{P}\right\}.$$
We shall show remaining pairs are always good and $\cup_{k\geq k_0} H_k$ is small to $E\times F$.

First we show $$\mu\left(T(l_{x,y}, \delta)\right)\leq \delta^{\eta}$$
for any remaining pair $(x,y)$. It is easy from our construction. If it is bad, there exists a $T_{y,M}$ satisfying (2a), (2b). Since $\textbf{BadBad}_\textbf{P}$ has been eliminated, $T(l_{x,y},\delta)$ must lie in the $\delta^{\rho/2}$-neighborhood of $l_{y, M}$. But it is impossible because $T(l_{y, M},\delta^{\rho/2})\times\{y\}$ has been removed.

Now we show $\cup_{k\geq k_0} H_k\ll 1$ when $k_0$ is large enough. 

Denote $\Gamma=\Gamma(\epsilon, {\rho})$ such that ${\rho}/2=(1+\epsilon)^{-\Gamma}$. Then the radius of tubes in $H_k$ becomes $\delta_{k-\Gamma}$. Notice both $\epsilon$ and ${\rho}$ are independent in $k$, so is $\Gamma$. 

For any $k_0\leq k\leq k_0+\Gamma$, there is no trick so we use the trivial bound
$$\mu\times\nu\left(\bigcup_{k_0\leq k\leq k_0+\Gamma} H_k\right)\leq \Gamma\cdot \sup_l\mu(T(l, \delta_{k_0-\Gamma})) $$

When $k\geq k_0+\Gamma$, instead of looking at each $H_k$, we consider $H_k\backslash\cup_{j<k}H_j$. Since $H_{k-\Gamma}$ is removed, all remaining pairs are good in the scale $\delta_{k-\Gamma}$, thus
$$\mu\times\nu\left(H_k\backslash\cup_{j<k}H_j\right)\leq \delta_{k-\Gamma}^\eta.$$

Above all,
$$\mu\times\nu\left(\bigcup_{k\geq k_0} H_k\right)\leq \Gamma\cdot \sup_l\mu(T(l, \delta_{k_0-\Gamma}))+\delta_{k_0-\Gamma}^\eta.$$
Since $\Gamma, \eta>0$ are fixed constants independent in $k_0$, by Lemma \ref{lemma-not-in-line} 
$$\mu\times\nu\left(\bigcup_{k\geq k_0} H_k\right)\ll 1 $$
when $k_0$ is large enough.

\subsection{}\label{tube-condition-step 4}
One can easily choose $\rho, \eta$ to ensure that \eqref{const-track-1-1}, \eqref{const-track-1-2} hold. Hence the proof is complete with
$$G:=\left(E\times (F\backslash \textbf{BadBad}_\textbf{P})\right)\backslash\cup_{k\geq k_0}H_k.$$

\end{proof}

\section{Proof of Theorem \ref{main-thm-pin-radial-proj}}
Suppose $\dH S(E)<\sigma$. Then by Lemma \ref{discretize-fractal} there exists $\epsilon>0$ such that for any $k_0>0$ there exists a family of $(\delta_k(\epsilon),\sigma)$-sets $X^\sigma_k$, $k\geq k_0$, that cover $S(E)$. Without loss of generality we work on $E, F$, $\dist(E, F)>0$, $\dH E=\dH F$, and only consider directions determined by pairs $(x, y)\in E\times F$. Let $\mu, \nu$ be Frostman measures on $E, F$ with $s_\mu=s_\nu=s$, and $G$ be as in Proposition \ref{prop-tube-condition}.

We shall find $\epsilon_0=\epsilon_0(s)>0$ such that, for any $\sigma<\frac{\dH E}{2}+\epsilon_0$, there exists $\beta>0$ such that for any $(\delta, \sigma)$-set $X\subset S^1$,
\begin{equation}\label{summable}\mu\times\nu\{(x,y)\in G:S(x,y)\in X\}\lesssim\delta^\beta.\end{equation}
It it holds, then for any $k_0>0$,
$$0<\mu\times\nu(G)\leq \sum_{k\geq k_0}\mu\times\nu\{(x,y)\in G:S(x,y)\in X_k^\sigma\}\leq C \sum_{k\geq k_0}\delta_k^\beta, $$
a contradiction. Therefore $\dH S(E)\geq \frac{\dH E}{2}+\epsilon_0$.

It remains to prove \eqref{summable}. By Cauchy-Schwarz a couple of times, it suffices to consider
$$\mu^2\times\nu^4\{(x_1, x_2, y_1, y_2, y_3, y_4): (x_i, y_j)\in G,\, S(x_i, y_j)\in X,\, i=1,2,\, j=1,2,3,4\}.$$
Notice that each $x_i$ is contained in four $(\delta, \sigma)$-pencils with tips $y_j$, and each $y_j$ is contained in two pencils with tips $x_i$.

Let $0<\rho\ll 1\ll n<\infty$ be positive constants that will be determined later.

First it suffices to consider $\dist (x_1, x_2), \dist (y_j, y_{j'})\geq \delta^{\rho/n}, \forall \, j\neq j'$, otherwise \eqref{summable} would follow from the ball condition on Frostman measures. 

By our tube condition Proposition \ref{prop-tube-condition}, we may assume triples $x_1, x_2, y_j$, as well as triples $x_i, y_j, y_{j'}$, do not lie in a tube of radius $\delta^{2\rho/n}$, equivalently $$\dist (l_{x_1, x_2}, y_j), \,\dist (l_{y_j, y_{j'}}, x_i)> \delta^{\rho/n}.$$

From now we fix $y_1, y_2$.

If one of $y_3, y_4$ does not lie in $T(l_{y_1, y_2}, \delta^{\rho})$, we fix $y_3, y_4$ as well. The rescaled version of our pencil estimate Theorem \ref{main-pencil} implies that $x_i$ is contained in the union of $\lesssim \delta^{-2\sigma+\epsilon_0-C\rho}$ balls of radius $\delta$. 
Therefore by the ball condition on $\mu$, the measure of each $x_i$ is
\begin{equation}\label{balance-1}\lesssim \delta^{s-2\sigma+\epsilon'_0(\sigma)-C\rho}.\end{equation}

It remains to consider the case $y_3, y_4\in T(l_{y_1, y_2}, \delta^{\rho})$. Fix $x_1, x_2$, then $y_3, y_4$ lie in the intersection of two $(\delta, \sigma)$-pencils $P_{x_1}, P_{x_2}$. From previous steps we have $\dist (x_1, x_2)\geq \delta_k^{\rho/n}$ and $\dist (l_{x_1, x_2}, y_j)> \delta^{\rho/n}$, $j=1,2,3,4$. Therefore $l_{x_1, y_j}$ and $l_{x_2, y_j}$ are always $\delta^{2\rho/n}$-separated. Then the non-concentration condition on our pencils guarantees that
$$T(l_{y_1, y_2}, \delta^{\rho})\cap P_{x_1}\cap P_{x_2}$$
can be covered by $\lesssim \delta^{-2\sigma+(\rho-2\rho/n)\sigma}$ balls of radius $\delta^{1-2\rho/n}$. See the figure below. Therefore the measure of each $y_j$, $j=3, 4$, is
\begin{equation}\label{balance-2}\lesssim \delta^{s-2\sigma+ \rho\sigma- 2(\sigma+s)\rho/n}.\end{equation}
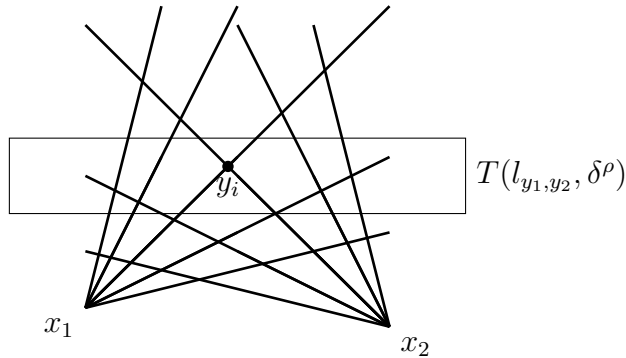
\begin{figure}[H]
\centering
\begin{tikzpicture}[scale=0.25, p2/.style={line width=0.275, black}, p3/.style={line width=0.15, black!50!white}]
\draw (0,0) node[anchor=north east]{$x_1$};
\draw (16,-1) node[anchor=north west]{$x_2$};
\draw (-4, 5) rectangle (20, 9);
\draw (20, 7) node[anchor = west]{$T(l_{y_1, y_2}, \delta^\rho)$};
\fill (7.5, 7.5) circle (0.3) node[anchor = north]{$y_i$};
\foreach \x in {2,...,4}{
\foreach \y in {2,...,4}{ 
	\draw[line width = 1] (0,0) -- (2^\x, 2^\y);\draw[line width = 1] (16,-1) -- (-2^\x+16, 2^\y-1); }}
\end{tikzpicture}
\end{figure}
Hence \eqref{summable} follows with a desired $\epsilon_0=\epsilon_0(s)>0$, by choosing $n$ large and $\rho>0$ to balance \eqref{balance-1}, \eqref{balance-2}.

\section{Proof of Theorem \ref{main-thm-pin-radial-proj}}
We may assume $F$ does not lie in a line, otherwise $(ii)$ in Theorem \ref{main-thm-pin-radial-proj} always holds.

Now we consider $\dH\pi^y(E)$. The idea is the same as the last section, with the $(\delta, \sigma)$-set $X$ in \eqref{summable} replaced by a family of $(\delta, \sigma)$-sets $X_y$. The only obstacle for this pinned version is, when $y_1, y_2, y_3, y_4$ lie in a tube of radius $\delta_k^\rho=\delta_{k-\Gamma}$ that is away from $E$, there is no way to control the measure of this tube. If we know this tube has measure $\lesssim \delta_{k-\Gamma}^\beta$ for any $k\geq k_0$, the pinned version follows. Otherwise we have a sequence of heavy tubes in $F$ away from $E$, which gives 
$$\overline{\dim}_{\mathcal{M}}\, \pi^x(F)= \dH F,\ x\in E.$$

\section{An alternative proof of \eqref{Orponen-half}}\label{app-tube-condition} 
Suppose $E, F$ are Borel sets in the plane, $\dH E, \dH F>0$, associated with Frostman measures $\mu,\nu$ respectively, and $\mu(l)=0$ for any line $l$. We shall show that $\dH \pi^y (E)\geq \frac{\dH E}{2}$ for some $y\in F$. 

As in \cite{Orp19} and many others, it suffices to show that, for any family of $\delta$-arcs $I_{y, i}\subset S^1$, $i=1,2,\dots, \delta^{-\sigma}$, $\sigma<\frac{\dH E}{2}$, there exists $\beta>0$ such that
$$\mu\times\nu\{(x,y)\in E\times F: \pi^y(x)\in \bigcup_i I_{y,i}\}\lesssim\delta^\beta.$$

By Cauchy-Schwarz it suffices to consider
$$\mu\times\nu\times\nu\{(x,y_1, y_2)\in E\times F\times F: \pi^{y_j}(x)\in \bigcup_i I_{y_j,i},\,j=1, 2\}.$$

When $F$ lies in a line, we may assume $\mu$ is away from this line. It is then trivial: if $y_1, y_2$ are close, it follows from the ball condition on $\nu$, otherwise it follows from the transversality between two pencils centered at $y_1, y_2$. A similar argument shows that \eqref{Bourgain-half} is trivial as well.

Suppose $F$ does not lie in a line, then by Proposition \ref{prop-tube-condition} it suffices to consider $(x, y)\in G$ from the beginning and look at
$$\mu\times\nu\times\nu\{(x,y_1, y_2): (x, y_j)\in G,\  \pi^{y_j}(x)\in \bigcup_i I_{y_j,i},\, j=1, 2\}.$$
By the ball condition on $\nu$, we may assume $y_1, y_2$ are separated. If $x$ lies in the $\delta^{\rho/n}$-neighborhood of $l_{y_1, y_2}$, it follows from the tube condition, otherwise it follows from the transversality between two pencils centered at $y_1, y_2$. 

Notice in this argument we do not need any non-concentration condition on $\cup_i I_{y, i}$.

\bibliographystyle{alpha}
\bibliography{/Users/MacPro/Dropbox/Academic/paper/mybibtex.bib}

\end{document}